\documentclass[12pt, leqno]{article}
\usepackage{amsmath,amsfonts,amsthm,amscd,amssymb,verbatim} 
\usepackage[all]{xy}

\numberwithin{equation}{section}

\theoremstyle{plain}
\newtheorem{thm}{Theorem}[section]
\newtheorem{defn}[thm]{Definition}
\newtheorem{prop}[thm]{Proposition}
\newtheorem{lem}[thm]{Lemma}
\newtheorem{cor}[thm]{Corollary}
\newtheorem{conj}[thm]{Conjecture}

\theoremstyle{definition}
\newtheorem{rem}[thm]{Remark}

\newcommand{\N}{{\Bbb N}}
\newcommand{\Z}{{\Bbb Z}}
\newcommand{\Q}{{\Bbb Q}}

\newcommand{\F}{{\Bbb F}}
\newcommand{\V}{{\Bbb V}}

\newcommand{\Spec}{{\mathrm{Spec}}\,}

\newcommand{\Spf}{{\mathrm{Spf}}\,}

\newcommand{\lra}{\longrightarrow}

\newcommand{\hra}{\hookrightarrow}

\newcommand{\isom}{\overset{\sim}{=}}

\newcommand{\wt}[1]{\widetilde{#1}}

\newcommand{\ol}[1]{\overline{#1}}
\newcommand{\os}{\overset}

\newcommand{\dR}{{\mathrm{dR}}}

\newcommand{\et}{{\mathrm{et}}}

\newcommand{\Hom}{{\mathrm{Hom}}}

\newcommand{\Ker}{{\mathrm{Ker}}}
\newcommand{\im}{{\mathrm{Im}}}
\newcommand{\Coker}{{\mathrm{Coker}}}

\newcommand{\id}{{\mathrm{id}}}

\newcommand{\pr}{{\mathrm{pr}}}

\newcommand{\bE}{{\bold E}}
\newcommand{\bX}{{\bold X}}

\newcommand{\cD}{{\cal D}}
\newcommand{\cE}{{\cal E}}
\newcommand{\cF}{{\cal F}}
\newcommand{\cG}{{\cal G}}
\newcommand{\cH}{{\cal H}}
\newcommand{\cI}{{\cal I}}

\newcommand{\cO}{{\cal O}}

\newcommand{\cU}{{\cal U}}

\newcommand{\cX}{{\cal X}}

\newcommand{\fe}{{\frak e}}

\newcommand{\fE}{{\frak E}}

\newcommand{\fG}{{\frak G}}

\newcommand{\fU}{{\frak U}}
\newcommand{\fX}{{\mathfrak{X}}}

\newcommand{\rig}{{\mathrm{rig}}}

\newcommand{\wh}{\widehat}
\renewcommand{\wt}{\widetilde}

\newcommand{\pa}{\partial}

\newcommand{\e}{{\bold e}}

\begin{document}
\title{A note on convergent isocrystals on simply connected varieties}
\author{
Atsushi Shiho\footnote{
Graduate School of Mathematical Sciences, 
University of Tokyo, 3-8-1 Komaba, Meguro-ku, Tokyo 153-8914, Japan. 
E-mail address: shiho@ms.u-tokyo.ac.jp. 
Mathematics Subject Classification (2010): 14F10, 14D20.}}
\date{}
\maketitle

\begin{abstract}
It is conjectured by de Jong that, 
if $X$ is a connected projective smooth variety over an algebraically closed 
field $k$ of characteristic $p>0$ with trivial etale fundamental group, 
any convergent isocrystal $\cE$ on 
$X$ is trivial. We discuss this conjecture when $X$ is liftable to 
characteristic zero, and 
prove the triviality of $\cE$ in this case 
under certain conditions on (semi)stability. 
\end{abstract}

\tableofcontents

\section*{Introduction}

For a connected 
quasi-projective smooth variety $X$ over an algebraically closed field 
$k$ of characteristic zero, it is proved by Malcev \cite{malcev} and 
Grothendieck \cite{grothendieck} that, 
if the etale fundamental group $\pi_1^{\et}(X)$ is trivial, 
any coherent $\cO_X$-module with integrable connection on $X$ over $k$ 
(which is equivalent to an $\cO_X$-coherent $\cD_X$-module on $X$) 
is trivial (isomorphic to a direct sum of $(\cO_X,d)$). 
The latter assertion is equivalent to the triviality of 
the de Rham fundamental group $\pi_1^{\dR}(X)$ defined as the Tannaka dual 
of the category of coherent $\cO_X$-modules with integrable connection 
on $X$ over $k$. So the above theorem can be interpreted as 
an interesting relation between $\pi_1^{\et}(X)$ and $\pi_1^{\dR}(X)$. 

As an analogue of this theorem in chatacteristic $p>0$, 
Esnault and Mehta \cite{esnaultmehta}, \cite{esnaultmehta2} 
proved a conjecture of 
Gieseker which says, 
for a connected 
projective smooth variety $X$ over an algebraically closed field 
$k$ of characteristic $p>0$ with trivial etale fundamental group, 
any stratified bundle on $X$ (which is equivalent to an $\cO_X$-coherent $\cD_X$-module on $X$) is trivial.  (In \cite{esnaultsrinivas}, the same statement 
for connected 
quasi-projective smooth $X$ is also proven under certain assumption.) 
The latter assertion is 
equivalent to the triviality of the stratified fundamental group 
$\pi_1^{\rm strat}(X)$ defined as the Tannaka dual 
of the category of stratified bundles on $X$, and the theorem 
reveals an interesting relation 
between $\pi_1^{\et}(X)$ and $\pi_1^{\rm strat}(X)$. 

As a $p$-adic version of the above 
theorem, there is the following conjecture, which is raised by de Jong 
according to private communication of the author with Esnault: 

\begin{conj}[de Jong]
For a connected 
projective smooth variety $X$ over an algebraically closed field 
$k$ of characteristic $p>0$ with trivial etale fundamental group, 
any convergent isocrystal on $X$ over $K$ $($where $K$ is the fraction 
field of a complete discrete valuation ring $V$ of mixed characteristic with 
residue field $k)$ is trivial. 
\end{conj}

As in the previous cases, the conclusion of the conjecture is 
equivalent to the triviality of the convergent fundamental group 
$\pi_1^{\rm conv}(X)$ defined as the Tannaka dual 
of the category of convergent isocrystals on $X$. 

In this paper, we discuss this conjecture when $X$ is liftable to 
a projective smooth scheme $\fX$ over $\Spec V$ 
(and $V$ admits a lift of absolute Frobenius on $k$).  
When the etale fundamental group $\pi_1^{\et}(\ol{\bX})$ 
of the generic geometric fiber $\ol{\bX}$ of $\fX$ is trivial, the conjecture 
in this case easily follows from 
the aforementioned result of Malcev and Grothendieck. 
But this does not imply the conjecture in general liftable case because 
we do not know in general whether $\pi_1^{\et}(\ol{\bX})$ is trivial or not
(although we know the triviality of its prime-to-$p$ quotients). 

We prove theorems (Theorem \ref{thm:1} and \ref{thm:2}) which 
roughly claim that, when $\cE$ admits very nice `mod $\varpi$ reductions' 
(where $\varpi$ is a uniformizer of $V$) 
which are $\mu$-stable or semistable sheaves on $X$, $\cE$ is trivial. 
As corollaries, 
we prove the triviality of $\cE$ when $\cE$ is of rank $1$ 
(Corollary \ref{cor:1}) or when $\cE$ admits a structure of 
a convergent $F$-isocrystal with `strongly semistable mod $\varpi$ reduction' 
(Corollary \ref{cor:2}, \ref{cor:22}). 

The rough idea of the proof is the following: 
First we prove the triviality of `$\cE$ modulo $\varpi$' 
using the result of Esnault-Mehta \cite{esnaultmehta} 
and 
the moduli of stable sheaves on $X$ constructed by Adrian Langer 
\cite{al1}, \cite{al2}. Then, we prove 
the triviality of `$\cE$ modulo $\varpi^N$' ($N \in \N$), which is 
a deformation of `$\cE$ modulo $\varpi$', 
by checking the behavior of deformation class under 
the level raising Frobenius pullback functor of Berthelot 
\cite{berthelotII}. Finally, we prove the triviality of $\cE$ 
by using Langton's theorem \cite{langton} or 
the moduli of semistable sheaves on $\fX$ 
constructed by Langer. 

The author is grateful to H\'el\`ene Esnault, because 
he started this study after the private communication with 
her on the above conjecture of de Jong. 
He is partly supported by JSPS 
Grant-in-Aid for Scientific Research (C) 25400008 and  
Grant-in-Aid for Scientific Research (B) 23340001. 

\section{Preliminaries and statement of results}
Throughout this paper, let $p$ be a fixed prime number, 
$k$ an algebraically closed field of characteristic $p$, 
$V$ a complete discrete valuation ring of mixed characteristic 
with residue field $k$ and $K$ the fraction field of 
$V$. Let $\varpi$ be a uniformizer of $V$. 
Denote the absolute ramification index of $V$ by $e$ and 
let $\fe$ be the minimal natural number with 
$e \leq p^{\fe}(p-1)$. 
We denote the absolute Frobenius morphism on 
$\Spec k$ by $\sigma$. 
We assume that there exists a lift 
$\sigma_V: \Spec V \lra \Spec V$ 
of $\sigma$ and fix it. 
We denote the $p$-adic completion 
$\Spf V \lra \Spf V$ of $\sigma_V$ by the same symbol. 
Denote the morphism $\Spec K \lra \Spec K$ induced by 
$\sigma_V$ by $\sigma_K$. 
Note that, when $V$ is the Witt ring $W(k)$ of $k$, the 
condition on the lift of Frobenius is satisfied and 
$\fe = 0$. 

Let $X$ be a projective smooth variety of dimension 
$d$ over $\Spec k$ and 
let $\cX$ be a projective smooth formal scheme over $\Spf V$ 
with $\cX \otimes_V k = X$. (Note that, throughout this paper, 
we assume the existence of such $\cX$.) 
By GFGA, $\cX$ naturally corresponds to 
a projective smooth scheme over $\Spec V$, which we denote by 
$\fX$. 
We denote its generic fiber by 
$\bX$ and its geometric generic fiber by $\ol{\bX}$. 
We assume that $X$ and $\ol{\bX}$ are connected. \par 

We denote the absolute Frobenius $X \lra X$ of $X$ by $F_{\rm abs}$. 
For $n \in \N$, 
denote $X \times_{\Spec k, \sigma^n} \Spec k$ 
(resp. $\fX \times_{\Spec V, \sigma_V^n} \Spec V$) by 
$X^{[n]}$ (resp. $\fX^{[n]}$) and the $p$-adic completion of 
$\fX^{[n]}$ by $\cX^{[n]}$. 
$X^{[n]}$ (resp. $\fX^{[n]}$, $\cX^{[n]}$) is isomorphic 
to $X$ (resp. $\fX$, $\cX$) 
as schemes (resp. schemes, formal schemes), but we prefer to 
use this notation to avoid confusion. 

Denote the projection $X^{[i+1]} \lra X^{[i]} \, (i \in \N)$ by 
$\pi$ and denote the $n$-times iteration of 
projections $X^{[i+n]} \lra X^{[i]} \, (i \in \N)$ by $\pi^n$. Also, 
denote the 
relative Frobenius morphism $X^{[i]} \lra X^{[i+1]} \, (i \in \N)$ 
by $F$ and denote the $n$-times iteration of 
 the relative Frobenius morphisms $X^{[i]} \lra X^{[i+n]} \, (i \in \N)$ 
by $F^n$. Then $F^n \circ \pi^n = F_{\rm abs}^n$ 
and $F$ and $\pi$ are `commutative' in suitable sense as long as 
they are defined. 
Denote also the projection $\fX^{[i+1]} \lra \fX^{[i]} \, (i \in \N)$ by 
$\pi$ and denote the $n$-times iteration of 
projections $\fX^{[i+n]} \lra \fX^{[i]} \, (i \in \N)$ by $\pi^n$. 
However, we do not assume the existence of a lift of 
a relative or absolute Frobenius morphism on $X^{[i]}$ to that on 
$\fX^{[i]}$. 

We fix an ample line bundle $\cO_{\fX}(1)$ on $\fX$. By restriction, 
this induces an ample line bundle on $X, \bX$, which we denote by 
$\cO_{X}(1)$, $\cO_{\bX}(1)$, respectively. 
Also, by the pullback by $\pi^n$, this induces 
an ample line bundle on $X^{[n]}, \bX^{[n]}$, which we denote by 
$\cO_{X^{[n]}}(1)$, $\cO_{\bX^{[n]}}(1)$, respectively. 
When we consider the slope $\mu(\cF)$ or the reduced 
Hilbert polynomial $p_{\cF}$ of 
a torsion free sheaf $\cF$ on $X^{[n]}$ (resp. $\bX^{[n]}$), we 
always consider them 
with respect to $\cO_{X^{[n]}}(1)$ (resp. $\cO_{\bX^{[n]}}(1)$). 
(For the definition of $\mu(\cF)$ and $p_{\cF}$, see 
\cite[Definition 1.2.11, 1.2.3]{hl}.) 
Therefore, when we consider (semi)stability 
(also known as Gieseker (semi)stability) 
and $\mu$-(semi)stability, we consider them 
with respect to $\cO_{X^{[n]}}(1)$ (resp. $\cO_{\bX^{[n]}}(1)$).

We give a review of rings of $p$-adic differential operators 
(arithmetic $D$-modules) 
defined by Berthelot \cite{berthelotI}, in our setting. 
Let $\cX$ be as above and let 
$\cD(1)_{(m)}$ (resp. $\cD(2)_{(m)}$) be the 
$p$-adically 
completed $m$-PD-envelope of $\cX$ in 
$\cX \times_{V} \cX$ (resp. 
$\cX \times_{V} \cX \times_V \cX$). 
Then $\cI := \Ker (\cO_{\cD(1)_{(m)}} \lra \cO_{\cX})$ is 
endowed with an $m$-PD-structure and so we can define 
the ideal $\cI^{\{ n \}} \,(n \in \N)$ as in \cite[1.3.7]{berthelotI}. 
We denote by $\cD(1)_{(m)}^n$ the closed formal subscheme of 
$\cD(1)_{(m)}$ defined by $\cI^{\{ n \}}$. 
Then we define $\wh{\cD}^{(m)}_{\cX}$ by 
$\cD^{(m)}_{\cX/p^N,n} := \Hom_{\cO_{\cX}}(\cO_{\cD(1)_{(m)}}/p^N\cO_{\cD(1)_{(m)}}, \cO_{\cX}/p^N \cO_{\cX}), 
\cD^{(m)}_{\cX/p^N} := \bigcup_{n \in \N} \cD^{(m)}_{\cX/p^N,n}$ and 
$\wh{\cD}^{(m)}_{\cX} := \varprojlim_{N \in \N}\cD^{(m)}_{\cX/p^N}. $ 
The morphism 
$$ \cD(1)_{(m)} \times_{\cX} \cD(1)_{(m)} \cong 
\cD(2)_{(m)} \longrightarrow \cD(1)_{(m)}$$ 
(the isomorphism follows from the explicit description of 
$m$-PD envelope given in \cite[1.5]{berthelotI}) 
induced by the projection $\cX \times_V \cX \times_V \cX 
\lra \cX \times_V \cX$ to the first and the third factors 
naturally defines the $\cO_{\cX}$-algebra 
structure on $\wh{\cD}^{(m)}_{\cX}$. 
We put $\wh{\cD}^{(m)}_{\cX,\Q} := \Q \otimes_{\Z} \wh{\cD}^{(m)}_{\cX}$ and 
finally we define $\cD_{\cX,\Q}^{\dagger}$ by 
$\cD_{\cX,\Q}^{\dagger} = \varinjlim_m \wh{\cD}_{\cX,\Q}^{(m)}$. 
Note that 
the above construction works also when $\cX$ is replaced by 
$\cX^{[i]} \,(i \in \N)$.

Let $\cE$ be a convergent isocrystal on $X/K$. 
In terms of arithmetic $D$-modules, it is nothing but 
an 
$\cO_{\cX,\Q}$-coherent $\cD_{\cX,\Q}^{\dagger}$-module 
$\cE$, where 
$\cO_{\cX,\Q} := \Q \otimes_{\Z} \cO_{\cX}$ 
(\cite[4.1.4]{berthelotI}). 
Let us denote the coherent $\cO_{\bX}$-module corresponding to $\cE$ 
via GFGA by $\bE$. (Then $\bE$ is known to be locally free.) 
For each $m \in \N$, $\cE$ naturally has a structure of 
$\cO_{\cX,\Q}$-coherent quasi-nilpotent 
$\wh{\cD}^{(m)}_{\cX,\Q}$-module. 
The next proposition assures the existence of a certain 
$\cO_{\cX}$-coherent lattice of $\cE$: 

\begin{prop}\label{prop:lattice}
Let the notations be as above. 
Then, there exists a $p$-torsion free 
$\cO_{\cX}$-coherent quasi-nilpotent 
$\wh{\cD}_{\cX}^{(m)}$-module $\cE^{(m)}$ with 
$\Q \otimes_{\Z} \cE^{(m)} = \cE$ as $\wh{\cD}_{\cX,\Q}^{(m)}$-modules. 
\end{prop}

To prove it, we recall the notion of 
$m$-HPD-stratification 
in our setting. 

\begin{defn}\label{mahah}
Let $\cX, \cD(i)_{(m)} \,(i=1,2)$ be as above. 
Also, let $p_i: \cD(1)_{(m)} \longrightarrow \cX \: (i=1,2)$, 
$p_{ij}: \cD(2)_{(m)} \longrightarrow \cD(1)_{(m)} \: (1 \leq i < j \leq 3)$ 
be the projections and let 
$\Delta: \cX \longrightarrow \cD(1)_{(m)}$ be the diagonal map. 
Then we define an $m$-HPD-stratification on a coherent $\cO_{\cX}$-module 
or a coherent $\cO_{\cX,\Q}$-module $\cE$ as an 
$\cO_{\cD(1)_{(m)}}$-linear isomorphism 
$\epsilon: p_2^*\cE \longrightarrow p_1^*\cE$ satisfying 
$\Delta^*(\epsilon) = \id$ and 
$p_{12}^*(\epsilon) \circ p_{23}^*(\epsilon) =
p_{13}^*(\epsilon)$. 
\end{defn}

Then, it is known (follows easily from \cite[2.3.7]{berthelotI}) 
that, for a coherent $\cO_{\cX}$-module (resp. 
a coherent $\cO_{\cX,\Q}$-module) $\cE$, giving 
a structure of quasi-nilpotent $\wh{\cD}^{(m)}_{\cX}$-module on $\cE$ 
(resp. a quasi-nilpotent 
$\wh{\cD}^{(m)}_{\cX,\Q}$-module on $\cE$) 
is equivalent to giving a structure of 
$m$-HPD-stratification on $\cE$. 
Therefore, to prove Proposition \ref{prop:lattice}, 
it suffices to prove the following: 

\begin{lem}\label{hahaha}
Let $\cE$ be a coherent $\cO_{\cX,\Q}$-module endowed with 
an $m$-HPD-stratification $\epsilon: p_2^*\cE \longrightarrow p_1^*\cE$. 
Then there exists a $p$-torsion free coherent 
$\cO_{\cX}$-module $\cF$ with $\Q \otimes_{\Z} \cF = \cE$ 
such that $\epsilon_{\cF} := \epsilon|_{p_2^*\cF}$ induces the 
HPD-stratification $p_2^*\cF \lra p_1^*\cF$ on $\cF$. 
\end{lem}
        
\begin{proof}
The proof is the level $m$ version of \cite[(0.7.4)]{ogus2}, which is based on 
the technique of the 
proof of rigid analytic faithfully flat descent due to O. Gabber 
(cf. \cite[(0.7.2)]{ogus2}, \cite[(1.9)]{ogus1}). 
In this proof, we denote the $p$-adically completed
 tensor product of modules by $\hat{\otimes}$ and use the symbol $\otimes$ 
only for usual tensor products. \par 
Take any $p$-torsion free coherent 
$\cO_{\cX}$-module $\cF'$ with $\Q \otimes_{\Z} \cF' = \cE$, and 
we define the map $\theta: \cE \longrightarrow p_1^*\cE$ by 
$\theta(x) = \epsilon(p_2^*x)$. 
By the description of $m$-PD envelope given in 
\cite[1.5]{berthelotI}, 
the maps $p_i:\cD(1)_{(m)} \longrightarrow \cX$ are flat.
 Hence $p_1^*\cF' \subset p_1^*\cE$. Put 
$\phi':= \theta \vert_{\cF'}$ and
 let $\cF$ be $\theta^{-1}(p_1^*\cF')$. \par
First check the inclusion $\cF \subseteq \cF'$. Let 
$\Delta_{\cE}: p_1^*\cE \longrightarrow \cE$
 be the map defined by 
$x\otimes \gamma \mapsto x\Delta^*(\gamma)$ for $x \in \cE$ and 
$\gamma \in {\cal O}_{\cD(1)_{(m)}}$. Then $\Delta_{\cE} \circ \theta = \id$ 
and $\Delta_{\cE}$ sends $p_1^*\cF'$ into $\cF'$. 
If $x$ is a local section of $\cF$,
 $\theta(x) \in p_1^*\cF'$ and so $ x = 
\Delta_{\cE} \circ \theta(x) \in \cF'$. 
Hence $\cF \subseteq \cF'$. \par 
Next we prove the coherence of $\cF$. 
For an affine open formal subscheme $\cU = \Spf A \subseteq \cX$, 
we denote $\Gamma(\cU,\cF), \Gamma(\cU,\cF'), \Gamma(\cU,\cE)$ by 
$F_A, F'_A, E_A$, respectively. Then, for any such $A$, 
$F'_A$ is a finitely generated module over the Noetherian ring $A$ and 
$F_A \subseteq F'_A$. So $F_A$ is a finitely generated 
$A$-module. Next, let us take affine open formal subschemes
$\cU' = \Spf A' \subseteq \cU = \Spf A \subseteq \cX$, and 
denote $\cU \times_{\cX} \cD(i)_{(m)}$ 
(resp. $\cU' \times_{\cX} \cD(i)_{(m)}$), 
which is an affine open formal subscheme of $\cD(i)_{(m)}$, by 
$\Spf B(i)$ (resp. $\Spf B'(i)$). Then we have 
{\allowdisplaybreaks{
\begin{align*} 
F_{A'} & = \theta_{A'}^{-1}(F'_{A'} \otimes_{A'} (B(1) \hat{\otimes}_A A')) \\
& = 
\Ker(E_{A'} \os{\theta_{A'}}{\lra} 
E_{A'} \otimes_{A'} (B(1) \hat{\otimes}_A A') \lra 
(E_{A'}/F'_{A'}) \otimes_{A'} (B(1) \hat{\otimes}_A A')) \\ 
& = 
\Ker(E_{A'} \os{\theta_{A'}}{\lra} 
E_{A'} \otimes_{A'} (B(1) \hat{\otimes}_A A') \lra 
(E_{A'}/F'_{A'}) \otimes_{A'} (B(1) \otimes_A A'))  \\ 
& \hspace{10cm} \text{($E_{A'}/F'_{A'}$ is $p$-torsion)} \\ 
& = 
\Ker(E_A \otimes_A A' \os{\theta_A \otimes_A A'}{\lra} 
E_{A} \otimes_{A} B(1) \otimes_A A' \lra 
(E_{A}/F'_{A}) \otimes_{A} B(1) \otimes_A A') \\ 
& = 
\Ker(E_A \os{\theta_A}{\lra} E_A \otimes_A B(1) \lra 
(E_A/F'_A) \otimes_A B(1)) \otimes_A A' \\ 
& = 
\theta_A^{-1}(F'_A \otimes_A B(1)) \otimes_A A' 
= F_A \otimes_A A'. 
\end{align*}}}
So $\cF$ is coherent, as desired. \par 
Next we prove that $\theta(\cF)$ is contained in $p_1^*\cF$. 
Since $p_1$ is flat, $p_1^*\cF \subseteq p_1^*\cF'$.
Hence we have to show that the map $\phi = \theta \vert_{\cF}: \cF 
\longrightarrow p_1^*\cF'$ factors through $p_1^*\cF$. 
Since the assertion is local, we may check it on 
an open affine formal subscheme 
$\cU = \Spf A \subseteq \cX$. 
Let $F_A, F'_A, E_A, B(i)$ be as in the previous paragraph, and denote 
the composite 
$B(1) \os{p_{13}^*}{\longrightarrow} B(2) \cong B(1) \hat{\otimes}B(1)$ 
by $\delta$. 
Then, by cocycle condition, the following diagram is commutative:
\begin{equation}\label{keima}
\begin{CD}
 E_A @>\theta>> E_A \otimes_A B(1) \\
@V{\theta}VV @VV{\id \otimes \delta}V \\
E_A \otimes_A B(1) @>>{\theta \otimes \id}> 
E_A \otimes_A (B(1) \hat{\otimes} B(1)). 
\end{CD}
\end{equation}
  
Consider the following diagram:

\begin{equation*}
\begin{CD}
@. F_A \otimes_A B(1) @>{\phi \otimes \id}>> 
F'_A \otimes_A B(1) \otimes_A B(1) \\
@. @VVV @VVV \\
F_A @>\phi>> F'_A \otimes_A B(1) @>{\phi' \otimes \id}>> 
E_A \otimes_A B(1) \otimes_A B(1) \\
@VVV @VVV @\vert \\
E_A @>\theta>> E_A \otimes_A B(1) @>{\theta \otimes \id}>> 
E_A \otimes_A B(1) \otimes_A B(1). 
\end{CD}
\end{equation*}

By definition, the square on the bottom left is Cartesian. Since $p_1$
 is flat, the large rectangle on the right is also Cartesian. Thus it
 suffices to prove that the composition
 $$ F_A \rightarrow E_A \overset{\theta}{\rightarrow}
    E_A \otimes_A B(1) \overset{\theta \otimes \id}{\rightarrow}
    E_A \otimes_A B(1) \otimes_A B(1) 
    \overset{\pr \otimes \id}{\rightarrow} E_A/F'_A \otimes_A B(1) 
    \otimes_A B(1) $$
is the zero map. Since $E_A/F'_A$ is $p$-torsion, the natural map
  $$ E_A/F'_A \otimes_A B(1) \otimes_A B(1)
    \longrightarrow E_A/F'_A \otimes_A (B(1) \hat{\otimes}_A B(1))$$
 is isomorphic, so it suffices to show that our map becomes zero after
 we follow it with this isomorphism. If $x \in F_A$, then 
$\theta(x) \in F'_A \otimes_A B(1)$ holds, and so  
$(\theta \otimes \id)(\theta(x)) = (\id \otimes \delta)(\theta(x)) \in
 F'_A \otimes_A (B(1) \hat{\otimes}_A B(1))$ by the commutative diagram 
 (\ref{keima}). This proves the assertion that 
 $\theta(\cF)$ is contained in $p_1^*\cF$. \par 
By the above assertion, 
$\theta$ induces a morphism 
$\epsilon_{\cF} : p_2^*\cF \lra p_1^*\cF$ with $\Q \otimes 
\epsilon_{\cF} = \epsilon$. Finally we prove that 
$\epsilon_{\cF}$ is an isomorphism. 
To see this, we may work on 
an open affine formal subscheme 
$\cU = \Spf A \subseteq \cX$. 
Let $F_A, B(1)$ be as before and 
let us consider the morphism 
$\epsilon_A := \epsilon_{\cF} |_{\cU} : 
B(1) \otimes_A F_A \os{\cong}{\lra} F_A \otimes_A B(1)$. 
Let us put 
$C:=\Coker(\epsilon_A)$. 
It suffices to prove that 
$C/p^nC=0$ for any $n$. Note that $A \otimes_{\Delta^*,B(1)} \epsilon_A$ 
is, by definition, the identity map $F_A \lra F_A$. So we have 
$A \otimes_{\Delta^*,B(1)} C = 0$, hence 
$(A/p^nA) \otimes_{\Delta^*, B(1)/p^nB(1)} (C/p^nC)=0$. Since 
$\Ker(\Delta^*: B(1)/p^nB(1) \lra A/p^nA)$ is a nil-ideal and $C/p^nC$ is 
finitely generated, it implies that 
$C/p^nC = 0$. So $\epsilon_A$ is an isomorphism and we are done. 
\end{proof}

Let us go back to the situation before Proposition \ref{prop:lattice}. 
We can prove a slightly stronger assertion than Proposition \ref{prop:lattice}. 
\begin{prop}\label{prop:lattice2}
Let the notations be as above. 
Then, we can take $\cE^{(m)}$ in Proposition \ref{prop:lattice} 
so that $E^{(m)} = \cE^{(m)}/\varpi \cE^{(m)}$ is a torsion free 
$\cO_X$-module. 
\end{prop}

\begin{proof}
Let $\fE^{(m)}$ be the coherent $\cO_{\fX}$-module 
corresponding to $\cE^{(m)}$ via GFGA. 
Also, let $\cD^{(m)}_{\fX,n}$ be the coherent 
$\cO_{\fX}$-module corresponding to 
$\cD^{(m)}_{\cX,n}$ via GFGA and put 
$\cD^{(m)}_{\fX} := \bigcup_n \cD^{(m)}_{\fX,n}$. 
(Then it forms a ring.) Then $\fE^{(m)}$ has a structure 
of $\cD^{(m)}_{\fX}$-module. \par 
Since $\bE$ is locally free and $\cE^{(m)}$ is $p$-torsion free, 
$\fE^{(m)}$ is locally free on an open 
subscheme $\fU$ of $\fX$ with ${\rm codim}(\fX \setminus \fU, \fX) \geq 2$. 
Let $j: \fU \hra \fX$ be the inclusion. 
Then, if we replace $\fE^{(m)}$ by $j_*j^*\fE^{(m)}$, 
it still has a structure of $\cD^{(m)}_{\fX}$-module and 
it is a reflexive $\cO_{\fX}$-module. So it 
satisfies Serre's condition $S_2$. Then 
$E^{(m)} = \fE^{(m)}/\varpi \fE^{(m)}$ satisfies 
Serre's condition $S_1$ and so it is torsion free. 
Then, the $p$-adic completion $\cE^{(m)}$ of $\fE^{(m)}$ 
satisfies the condition of the proposition, because 
the $p$-adic completion of $\cD^{(m)}_{\fX}$ is equal to 
$\wh{\cD}^{(m)}_{\cX}$. 
\end{proof}

Let the situation be as before Proposition \ref{prop:lattice} and 
take a $p$-torsion free 
$\cO_{\cX}$-coherent quasi-nilpotent 
$\wh{\cD}_{\cX}^{(m)}$-module 
$\cE^{(m)}$ with 
$\Q \otimes_{\Z} \cE^{(m)} = \cE$ as in 
Proposition \ref{prop:lattice2}. 
Also, let $\fE^{(m)}$, $E^{(m)}$ be as in (the proof of) 
Proposition \ref{prop:lattice2}. 

By Berthelot's theory Frobenius descent (\cite[2.3.7]{berthelotII}), 
there exist equivalences 
(called the level raising Frobenius pullback functor) 
\begin{equation}\label{eq:lrfp}
{\F^i}^*: 
\left( 
\begin{aligned}
& \text{$\cO_{\cX^{[i]}}$-coherent} \\ 
& \text{quasi-nilpotent $\wh{\cD}_{\cX^{[i]}}^{(m-i)}$-modules} 
\end{aligned} 
\right) 
\os{\cong}{\lra} 
\left( 
\begin{aligned}
& \text{$\cO_{\cX}$-coherent} \\ 
& \text{quasi-nilpotent $\wh{\cD}_{\cX}^{(m)}$-modules} 
\end{aligned} 
\right) 
\end{equation}
for $0 \leq i \leq m - \fe$ 
induced by $i$-times iteration $F^i$ of relative Frobenius. 
For precise definition of ${\F^i}^*$, 
see \cite[2.2.6(ii)]{berthelotII} or Section 2 in this paper. 
(We wrote the level raising Frobenius pullback functor
 by ${\F^i}^*$ to distinguish it from the pullback ${F^i}^*$ of 
$\cO_X$-modules by relative Frobenius. However, for 
$\cO_{\cX^{[i]}}$-coherent 
quasi-nilpotent $\wh{\cD}_{\cX}^{(m-i)}$-modules $\cF$ with 
$\varpi \cF = 0$, they are the same if we forget 
the $\wh{\cD}_{\cX}^{(m)}$-module structure on ${\F^i}^*\cF$.) 
For an $\cO_{\cX}$-coherent 
quasi-nilpotent $\wh{\cD}_{\cX}^{(m)}$-module $\cF$, 
we call the $\cO_{\cX^{[i]}}$-coherent 
quasi-nilpotent $\wh{\cD}_{\cX^{[i]}}^{(m-i)}$-module $\cF'$ 
satisfying ${\F^i}^*\cF' = \cF$ 
the $i$-th Frobenius antecedent of $\cE^{(m)}$. 

For $0 \leq i \leq m - \fe$, let $\cE^{(m)[i]}$, $E^{(m)[i]}$ be 
the $i$-th Frobenius antecedent of $\cE^{(m)}$, $E^{(m)}$, 
respectively. (Then $E^{(m)[i]} = \cE^{(m)[i]}/\varpi \cE^{(m)[i]}$.) 
Let us denote the coherent $\cO_{\fX^{[i]}}$-module 
corresponding to $\cE^{(m)[i]}$ via GFGA by $\fE^{(m)[i]}$. 
Then we have the following: 

\begin{prop}
The sheaf $\Q \otimes_{\Z} \cE^{(m)[i]}$ does not depend on  
the choice of $\cE^{(m)}$ as in Proposition \ref{prop:lattice}, and 
does not depend on $m \geq i + \fe$ either. $($Hence, by GFGA, 
the same is true for the restriction of $\fE^{(m)[i]}$ to $\bX^{[i]}.)$  
\end{prop}

\begin{proof}
If we choose another $\cE^{(m)}$ which we denote by 
$\cF^{(m)}$, we have morphisms 
$\alpha: \cE^{(m)} \lra \cF^{(m)}, \beta: \cF^{(m)} \lra \cE^{(m)}$ 
with $\beta \circ \alpha = p^N, \alpha \circ \beta = p^N$ for some 
$N \in \N$. Let us denote the $i$-th Frobenius antecedent 
of $\cF^{(m)}$ by $\cF^{(m)[i]}$.  
Then, since ${\F^i}^*$ is an equivalence, 
the maps $\alpha, \beta$ induce 
the maps $\alpha': \cE^{(m)[i]} \lra \cF^{(m)[i]}, 
\beta': \cF^{(m)[i]} \lra \cE^{(m)[i]}$ with 
$\beta' \circ \alpha' = p^N, \alpha' \circ \beta' = p^N$. 
So $\Q \otimes_{\Z} \cE^{(m)[i]} = \Q \otimes_{\Z} \cF^{(m)[i]}$. \par 
For $i+\fe \leq m \leq m'$, $\cE^{(m')}$ as in Proposition \ref{prop:lattice} 
can be regarded also as a $\wh{\cD}_{\cX}^{(m)}$-module 
via the restriction by the canonical 
map $\wh{\cD}_{\cX}^{(m)} \lra \wh{\cD}_{\cX}^{(m')}$, and 
the functor ${\F^i}^*$ is compatible with it. So we see that 
$\Q \otimes_{\Z} \cE^{(m)[i]}$ does not depend on $m$. 
\end{proof}

So, in the sequel, we denote the sheaf $\Q \otimes_{\Z} \cE^{(m)[i]}$ by 
$\cE^{[i]}$ and the restriction of 
$\fE^{(m)[i]}$ to $\bX^{[i]}$ by $\bE^{[i]}$. Note that these are 
defined for all $i \in \N$. Note also that 
$\cE^{[i]} = \Q \otimes \cE^{(m)[i]}$ has a structure of 
$\wh{\cD}^{(m)}_{\cX^{[i]},\Q}$-modules $(m \geq i + \fe)$ 
which are compatible 
with respect to $m$, which induces a structure of 
$\cD^{\dagger}_{\cX^{[i]},\Q}$-module. 

\begin{prop}\label{prop:hilb}
For any $i \in \N$, 
the reduced Hilbert polynomial $p_{\bE^{[i]}}$ of $\bE^{[i]}$ is 
equal to $p_{\cO_{X}}$. In particular, 
$\mu(\bE^{[i]}) = 0$. 
\end{prop}

\begin{proof}
We give two proofs. First, if we define $\cD_{\fX^{[i]}}^{(m)}$ as in the 
proof of Proposition \ref{prop:lattice2} (wuth $\fX$ replaced by 
$\fX^{[i]}$), $\fE^{(m)[i]}$ has a 
structure of $\cD_{\fX^{[i]}}^{(m)}$-module. Also, one can see that 
the restriction of $\cD_{\fX^{[i]}}^{(m)}$ to $\bX^{[i]}$ is nothing but 
the usual $D$-module $\cD_{\bX^{[i]}}$ on $\bX^{[i]}$. So 
$\bE^{[i]}$ has a structure of $\cD_{\bX^{[i]}}$-module. 
Then it is well-known (e.g., \cite[2.3.1]{kobayashi}) that 
the Chern classes of non-zero degree of $\bE^{[i]}$ vanish, and so  
$p_{\bE^{[i]}}$ is equal to $p_{\cO_{\bX^{[i]}}} = p_{\cO_{X^{[i]}}} = 
p_{\cO_X}$ by Riemann-Roch. \par 
The second proof is analogous to the proof in 
\cite[Lemma 2.1]{esnaultmehta}. 
By Riemann-Roch, it suffices to prove the following claim: 
For any $0 \leq j \leq d-1$, for 
any class $\xi \in CH_j(\bX^{[i]})$ and 
for any homogeneous polynomial $\gamma_{d-j}$ of degree $d-j$ 
with rational coefficients in the Chern class, 
we have $\xi_{\et} \cdot \gamma_{d-j}(\bE^{[i]})_{\et} = 0$,
where $\xi_{\et}, \gamma_{d-j}(\bE^{[i]})_{\et}$ are the 
class of $\xi, \gamma_{d-j}(\bE^{[i]})$ considered in 
$l$-adic etale cohomology 
$H^{*}_{\et}(\ol{\bX}^{[i]},\Q_l) \cong 
H^{*}_{\et}(X^{[i]},\Q_l)$. 
(Here $\ol{\bX}^{[i]}$ is the geometric fiber of $\bX^{[i]}$.) 
If we take any $i \leq i' \leq m - \fe$, 
we have 
\begin{align*}
\xi_{\et} \cdot \gamma_{d-j}(\bE^{[i]})_{\et} & = 
\xi_{\et} \cdot \gamma_{d-j}(E^{(m)[i]})_{\et} = 
\xi_{\et} \cdot \gamma_{d-j}({F^{i'-i}}^*E^{(m)[i']})_{\et}  
\\ & = 
p^{i'-i} 
\xi_{\et} \cdot \gamma_{d-j}(E^{(m)[i']})_{\et} 
= 
p^{i'-i} 
\xi_{\et} \cdot \gamma_{d-j}(\bE^{[i']})_{\et}. 
\end{align*}
Since the above equality holds for any $i \leq i'$ and 
$\xi_{\et} \cdot \gamma_{d-j}(\bE^{[i']})_{\et}$ \,$(i \leq i')$ 
are rational numbers with bounded denominator 
(depending only on $\xi$ and $\gamma_{d-j}$), 
we see the equality 
$\xi_{\et} \cdot \gamma_{d-j}(\bE^{[i]})_{\et} = 0$. 
\end{proof}

We say that a convergent isocrystal $\cE$ on $X/K$ is trivial 
if it is isomorphic to a finite direct sum of 
the structure convergent isocrystal $\cO_{X/K}$ 
(which corresponds to the $\cD^{\dagger}_{\cX,\Q}$-module 
$\cO_{\cX,\Q}$ defined by the canonical action via derivation). 

To proceed further, we have to impose certain conditions of 
(semi)stability. We state our first result, which assumes a 
certain stability condition: 

\begin{thm}\label{thm:1}
Let $X$ be as above and assume that its etale fundamental group 
$\pi_1^{\et}(X)$ is trivial. Then, any 
convergent isocrystal $\cE$ on $X/K$ which 
satisfies the following condition {\rm (A)} 
is trivial\,$:$ \\ 
{\rm (A)}$:$ For infinitely many natural numbers $m$, there exists 
some $i = i(m) \in \N$ and 
some $p$-torsion free 
$\cO_{\cX}$-coherent quasi-nilpotent 
$\wh{\cD}_{\cX^{[i]}}^{(m)}$-module $\cG^{[i](m)}$ with 
$\Q \otimes_{\Z} \cG^{[i](m)} = \cE^{[i]}$ 
as $\wh{\cD}_{\cX^{[i]},\Q}^{(m)}$-modules such that 
$G^{[i](m)} := \cG^{[i](m)}/\varpi \cG^{[i](m)}$ is 
$\mu$-stable as $\cO_{X^{[i]}}$-module. 
\end{thm}

\begin{cor}\label{cor:1}
Let $X$ be as above and assume that its etale fundamental group 
$\pi_1^{\et}(X)$ is trivial. Then any convergent isocrystal $\cE$ 
on $X/K$ of rank $1$ is trivial. 
\end{cor}

\begin{proof}
If the rank of $\cE$ is equal to $1$, 
$\cE^{(m)} \, (m \in \N)$ in Proposition \ref{prop:lattice2} 
satisfies the condition {\rm (A)} (for $\cG^{[i](m)}$ with $i=0$) 
because any torsion free $\cO_X$-module of generic rank $1$ is 
$\mu$-stable. 
\end{proof}

We have a similar result under 
certain assumption of semistability: 

\begin{thm}\label{thm:2}
Let $X$ be as above and assume that its etale fundamental group 
$\pi_1^{\et}(X)$ is trivial. Then any 
convergent isocrystal $\cE$ on $X/K$ 
which satisfies the following 
condition {\rm (B)} is trivial\,$:$ \\
{\rm (B)}$:$ For infinitely many natural numbers $m$, 
there exists some $i = i(m), l = l(m) 
\in \N$ with $l(m) \leq m - \fe$, 
$l(m) \to \infty \,\, (m \to \infty)$ 
and some $p$-torsion free 
$\cO_{\cX^{[i]}}$-coherent quasi-nilpotent 
$\wh{\cD}_{\cX^{[i]}}^{(m)}$-module $\cG^{[i](m)}$ with 
$\Q \otimes_{\Z} \cG^{[i](m)} = \cE^{[i]}$ 
as $\wh{\cD}_{\cX^{[i]},\Q}^{(m)}$-modules such that, for 
any $0 \leq j \leq l$, the $j$-th Frobenius antecedent 
$G^{[i](m)[j]}$ of 
$G^{[i](m)} := \cG^{[i](m)}/\varpi \cG^{[i](m)}$ is semistable as 
$\cO_X$-module. 
\end{thm}

We have the ${\sigma_{K}^n}^*$-linear 
pullback functor 
$$ 
{F_{{\rm abs}}^n}^*: 
\left( 
\begin{aligned}
& \text{convergent} \\ 
& \text{isocrystals on $X/K$}
\end{aligned}
\right) 
\os{\cong}{\lra} 
\left( 
\begin{aligned}
& \text{convergent} \\ 
& \text{isocrystals on $X/K$}
\end{aligned}
\right) 
$$ 
induced by $F_{{\rm abs}}^n$. 
In terms of arithmetic $D$-modules, 
it is written as the composite of 
the ${\sigma_K^n}^*$-linear pullback functor 
$$ 
{\pi^n}^*: \left( 
\begin{aligned}
& \text{$\cO_{\cX,\Q}$-coherent} \\ 
& \text{$\cD_{\cX}^{\dagger}$-modules} 
\end{aligned} 
\right) 
\os{\cong}{\lra} 
\left( 
\begin{aligned}
& \text{$\cO_{\cX^{[n]},\Q}$-coherent} \\ 
& \text{$\cD_{\cX^{[n]}}^{\dagger}$-modules} 
\end{aligned} 
\right) 
$$ 
induced by $\pi^n$ and the functor 
$$ {\F^n_{\Q}}^*: \left( 
\begin{aligned}
& \text{$\cO_{\cX^{[n]},\Q}$-coherent} \\ 
& \text{$\cD_{\cX^{[n]}}^{\dagger}$-modules} 
\end{aligned} 
\right) 
\os{\cong}{\lra} 
\left( 
\begin{aligned}
& \text{$\cO_{\cX,\Q}$-coherent} \\ 
& \text{$\cD_{\cX}^{\dagger}$-modules} 
\end{aligned} 
\right) 
$$ 
induced by the $\Q$-linearization of the 
level raising Frobenius pullback functors 
$$ 
{\F^n}^*: 
\left( 
\begin{aligned}
& \text{$\cO_{\cX^{[n]}}$-coherent} \\ 
& \text{quasi-nilpotent $\wh{\cD}_{\cX^{[n]}}^{(m)}$-modules} 
\end{aligned} 
\right) 
\os{\cong}{\lra} 
\left( 
\begin{aligned}
& \text{$\cO_{\cX}$-coherent} \\ 
& \text{quasi-nilpotent $\wh{\cD}_{\cX}^{(m+n)}$-modules} 
\end{aligned} 
\right)  
$$ 
for $m \gg 0$. Note that the functors 
${\pi^n}^*$, ${\F^n_{\Q}}^*$ can be written as a certain 
iteration of the functors ${\pi}^*$, ${\F_{\Q}}^*$ in the case
$n=1$, and they are `commutative' in suitable sense 
as long as they are defined. 

In this paper, a convergent $F$-isocrystal on $X/K$ is a pair 
$(\cE, \Phi)$ consisting of a convergene isocrystal 
$\cE$ on $X/K$ endowed with an isomorphism 
$\Phi: {F_{\rm abs}^n}^*\cE \os{\simeq}{\lra} \cE$ for some $n \in \N$. 

Then we have the following corollary of Theorem \ref{thm:2}: 

\begin{cor}\label{cor:2}
Let $X$ be as above and assume that its etale fundamental group 
$\pi_1^{\et}(X)$ is trivial. Then, for any 
convergent $F$-isocrystal $(\cE, \Phi)$ on $X/K$ 
which satisfies the following 
condition {\rm (C)}, $\cE$ is trivial as a convergent isocrystal\,$:$ \\
{\rm (C)}$:$ There exists some $i \in \N$ and 
some $p$-torsion free 
$\cO_{\cX^{[i]}}$-coherent quasi-nilpotent 
$\wh{\cD}_{\cX^{[i]}}^{(\fe)}$-module $\cG^{[i](\fe)}$ with 
$\Q \otimes_{\Z} \cG^{[i](\fe)} = \cE^{[i]}$ 
as $\wh{\cD}_{\cX^{[i]},\Q}^{(\fe)}$-modules such that 
$G^{[i](\fe)} := \cG^{[i](\fe)}/\varpi 
\cG^{[i](\fe)}$ is strongly semistable as 
$\cO_X$-module. 
\end{cor}

\begin{proof}
Because we have ${\F^i_{\Q}}^*\cE^{[i]} = \cE, \, 
{\F^i_{\Q}}^*{F_{\rm abs}^n}^*\cE^{[i]} = 
{\F^i_{\Q}}^*{\F^n_{\Q}}^*{\pi^n}^*\cE^{[i]} = 
{\F^n_{\Q}}^*{\pi^n}^*{\F^i_{\Q}}^*\cE^{[i]} = {F_{\rm abs}^n}^*\cE$, 
$\cE^{[i]}$ admits a structure of convergent $F$-isocrystal by some 
isomorphism $\Psi: {F_{\rm abs}^n}^* \cE^{[i]} \os{\cong}{\lra} \cE^{[i]}$. 
Then, for $m \in n \N + \fe$, we put 
$\cG^{[i](m)} := {\F^{m-\fe}}^*{\pi^{m-\fe}}^*\cG^{[i](\fe)}$, 
$G^{[i](m)} := \cG^{[i](m)}/\varpi \cG^{[i](m)}$. 
Then we have $\Q \otimes_{\Z} \cG^{[i](m)} = 
{F_{\rm abs}^{m-\fe}}^*\cE^{[i]} \cong \cE^{[i]}$ 
as $\wh{\cD}_{\cX^{[i]},\Q}^{(m)}$-modules. 
Also, for $0 \leq j \leq m - \fe$, the $j$-th Frobenius antecedent 
of $G^{[i](m)}$ is equal to ${F^{m-\fe - j}}^*{\pi^{m-\fe }}^*G^{[i](\fe)} 
$, which is semistable by the strong semistability of 
$G^{[i](\fe)}$. So 
$\cE$ satisfies the condition {\rm (B)} in Theorem \ref{thm:2} and so 
it is trivial as a convergent isocrystal. 
\end{proof}

We restate the corollary in the case $V = W(k), i=0$ for the convenience 
to the reader: 

\begin{cor}\label{cor:22}
Let $X$ be a projective smooth variety over $k$ which is liftable 
to a projective smooth formal scheme $\cX$ over $\Spf W(k)$. 
Also, assume that the etale fundamental group 
$\pi_1^{\et}(X)$ is trivial. Then, for any 
convergent $F$-isocrystal $(\cE, \Phi)$ on $X/{\rm Frac}\,W(k)$ 
which satisfies the following 
condition {\rm (D)}, $\cE$ is trivial as a convergent isocrystal\,$:$ \\
{\rm (D)}$:$ There exists some $p$-torsion free 
$\cO_{\cX}$-coherent quasi-nilpotent 
$\wh{\cD}_{\cX}^{(0)}$-module $\cG$ with 
$\Q \otimes_{\Z} \cG = \cE$ 
as $\wh{\cD}_{\cX,\Q}^{(0)}$-modules such that 
$G := \cG/p\cG$ is strongly semistable as 
$\cO_X$-module. 
\end{cor}

\section{Proofs}

In this section, we give proofs of Theorems 
\ref{thm:1}, \ref{thm:2}. 
So, in this section, let $X$ be as in the previous section 
and assume moreover that $\pi_1^{\et}(X)$ is trivial. 

\begin{prop}
To prove theorems, we can enlarge $k$ so that $k$ is uncountable.
\end{prop}

\begin{proof}
Let $k'$ be an uncountable algebraically closed field containing $k$ and 
let $K'$ be the fraction field of $V \otimes_{W(k)} W(k')$. 
Put $X' := X  \otimes_k k'$ and denote the pull-back of $\cE$ to 
$X'$ by $\cE'$, which is a convergent isocrystal on $X'/K'$. 
Also, put $\bX' := \bX \otimes_{K} K'$, let 
$\bE'$ be the pullback of $\bE$ to $\bX'$ and let 
$\cD_{\bX}, \cD_{\bX'}$ be usual $D$-modules of $\bX, \bX'$. 
Then, as we have seen in the proof of Proposition 
\ref{prop:hilb}, $\bE$ admits naturally a structure of 
$\cD_{\bX}$-module and $\bE'$ admits a structure of 
$\cD_{\bX'}$-module, which is the pullback of the $\cD_{\bX}$-module 
structure of $\bE$. \par 
If we assume that the theorem is true for $X'$ and $\cE'$, we have 
$\dim_{K'} H^0_{\rig}(X'/K',\cE') \allowbreak = r$. Because 
$$H^0_{\rig}(X'/K',\cE') = H^0_{\dR}(\bX'/K', \bE')
= H^0_{\dR}(\bX/K, \bE) \otimes_K K'
= H^0_{\rig}(X/K,\cE) \otimes_K K'$$ 
by the comparison theorem of Ogus (\cite{ogus1}, \cite{ogus2}), we have 
the equality $\dim_{K} H^0_{\rig}(X/K, \allowbreak 
\cE) = r$. So $\cE$ is also trivial. 
\end{proof}

So, in the sequel, we assume that $k$ is uncountable. 

The following proposition, which uses Gieseker conjecture 
(proven by Esnault-Mehta \cite{esnaultmehta}) in $\mu$-stable case, 
is the first step for the proof: 

\begin{prop}\label{prop:a} 
Let $r$ be a positive integer. Then there exists a positive integer 
$a = a(r)$ satisfying the following condition$:$ 
For any sequence of stable sheaves $\{E_i\}_{i=0}^a$ of length $a$ 
with $E_a$ on $X^{[a+j]}$ for some $j \geq 0$, ${\rm rank}\,E_0 \leq r$, 
$p_{E_i} = p_{\cO_X} \,(0 \leq i \leq a)$ and 
$F^*E_{i+1} = E_i \,(0 \leq i \leq a-1)$, 
$E_0$ is isomorphic to $\cO_{X^{[j]}}$. 
\end{prop}

\begin{proof} 
For $1 \leq s \leq r$ and $n \geq 0$, let $M_s^{[n+j]}$ be the moduli 
of stable sheaves on $X^{[n+j]}$ with rank $s$ and reduced 
Hilbert polynomial 
$p_{\cO_X}$, which is constructed by Adrian Langer (\cite{al1}, \cite{al2}). 
It is 
a quasi-projective scheme over $k$. Also, let $M_{s,\circ}^{[n+j]}$ be 
the open subscheme 
consisting of stable sheaves $G$ such that ${F^n}^*G$ remains 
stable. (This is known to be an open condition. See discussion in 
the beginning of \cite[\S 3]{esnaultmehta}.) 
The pull-back by $F$ induces the morphisms called Verschiebungs 
$$ \cdots \lra M_{s,\circ}^{[2+j]} \os{\V}{\lra} M_{s,\circ}^{[1+j]} \os{\V}{\lra} M^{[j]}_s. $$
Let $\im \V^n$ be the image of $\V^n: M_{s,\circ}^{[n+j]} \lra M_s^{[j]}$, 
which is 
a constructible set of $M_s^{[j]}$. Then, $\dim \im \V^n$ 
is stable for $n \gg 0$, which we denote by $f$. 
Assume $f > 0$. Then the generic point of 
some irreducible closed subscheme of dimension $f$ remains contained in 
$\im \V^n \, (n \in \N)$. 
Pick such an irreducible closed subscheme and denote it by $C$. 
Then $C \cap \im \V^n$ is non-empty for any $n \in \N$ and it contains 
an open subscheme of $C$. So there exists a closed 
subscheme $D_n \subsetneq C$ of smaller dimension 
such that $C \cap \im \V^n \supseteq C \setminus D_n$. Then 
$C \cap (\bigcap_n \im \V^n) \supseteq C \setminus (\bigcup_n D_n)$. 
So it contains at least two $k$-rational points $P,P'$, 
because $k$ is uncountable. 
Since $P,P'$ are $k$-rational points of $\bigcap_n \im \V^n$, 
they induce two non-isomorphic stratified sheaves on $X^{[j]}$. 
This contradicts (the $\mu$-stable case of) 
the Gieseker conjecture proven by 
Esnault-Mehta. So $\im \V^n$ consists of 
finite set of points (possibly empty) for some $n$. 
Then, since $\bigcap_n \im \V^n$ is empty (if $s \geq 2$) or 
one point corresponding to $\cO_{X^{[j]}}$ 
(if $s=1$) by (the $\mu$-stable case of) Gieseker conjecture, 
it is equal to $\im \V^{a(s)}$ for some $a(s) \in \N$. 
Let us define $a$ to be the maximum of $a(s) \, (s \leq r)$. 
Then, if we are given a 
sequence $\{E_i\}_{i=0}^a$ 
as in the statement of the proposition 
with $s := {\rm rank} E_0 \leq r$, 
$E_0$ defines a $k$-rational point of $\im \V^{a(s)} \subseteq M^{[j]}_s$. 
Then $s$ should be equal to $1$ and 
$E_0$ should be isomorphic to 
$\cO_{X^{[j]}}$. 
\end{proof}

\begin{rem}\label{error?}
There is a mistake in 
\cite[Proposition 2.3]{esnaultmehta}, but it is fixed in 
\cite{esnaultmehta2}. 
We also point out that this mistake occurs in the discussion of 
reducing Gieseker's conjecture to that in $\mu$-stable case. 
Because we used Gieseker's conjecture only in $\mu$-stable case, 
we do not need the correction given in \cite{esnaultmehta2}. 
\end{rem}

For the proof of Theorem \ref{thm:2}, 
we need the following proposition, which proves the triviality of 
locally free sheaves of higher rank in certain situation. 

\begin{prop}\label{prop:b}
There exists a positive integer 
$b$ satisfying the following condition$:$ 
For any sequence of locally free sheaves $\{E_i\}_{i=0}^{b(r-1)}$ of length 
$b(r-1)$ 
on $X$ with $E_{b(r-1)}$ on $X^{[b(r-1)+j]}$ for some $j \geq 0$, 
${\rm rank}\,E_0 = r$, 
$F^*E_{i+1} = E_i \,(0 \leq i \leq b(r-1)-1)$ such that 
$E_{b(r-1)}$ is an iterated extension of $\cO_{X^{[b(r-1)+j]}}$, 
$E_0$ is isomorphic to $\cO_{X^{[j]}}^r$. 
\end{prop}

\begin{proof}
The proof is similar to that in \cite[Proposition 2.4]{esnaultmehta}. 
By \cite[Corollary in p.143]{mumford}, we have the decomposition 
$H^1(X^{[n]},\cO_{X^{[n]}}) = H^1(X^{[n]},\cO_{X^{[n]}})_{\rm nilp} \oplus 
H^1(X^{[n]},\cO_{X^{[n]}})_{\rm ss}$ of $H^1(X^{[n]},\cO_{X^{[n]}})$ 
into the part 
$H^1(X^{[n]},\cO_{X^{[n]}})_{\rm nilp}$ 
where the absolute Frobenius $F_{\rm abs}^*$ 
acts nilpotently and the part 
$H^1(X^{[n]},\cO_{X^{[n]}})_{\rm ss}$ 
where the absolute Frobenius $F_{\rm abs}^*$ 
acts bijectively. Also, we have 
$$H^1(X^{[n]},\cO_{X^{[n]}})_{\rm ss} = H^1_{\et}(X^{[n]},\Z/p\Z) = 
\Hom(\pi_1^{\et}(X^{[n]}),\Z/p\Z) = 0$$ 
and there exists some $b \in \N$ 
such that $(F_{\rm abs}^*)^b$ acts by $0$ on 
$H^1(X^{[n]},\cO_{X^{[n]}})_{\rm nilp}$, 
since $H^1(X^{[n]},\cO_{X^{[n]}})_{\rm nilp}$ is finite-dimensional. 
So $(F_{\rm abs}^*)^b$ acts by $0$ on $H^1(X^{[n]},\cO_{X^{[n]}})$. 
(Also, we can take $b$ independently of $n \in \N$ because $X^{[n]}$ is 
isomorphic to $X$ via $\pi^n$.) 
We prove the proposition for this choice of $b$. 

By assumption on $E_{b(r-1)}$, there exists a filtration 
$$ 0 = E_{b(r-1),0} \subset E_{b(r-1),1} \subset  
\cdots \subset E_{b(r-1),r} = E_{b(r-1)} $$
whose graded quotients are isomorphic to $\cO_{X^{[b(r-1)+j]}}$. 
By pulling it back to $E_{i}$ via ${F^*}^{b(r-1)-i}$, we obtain the 
filtration 
$$ 0 = E_{i,0} \subset E_{i,1} \subset  \cdots \subset E_{i,r} = E_{i} $$
of $E_i$ whose graded quotients are isomorphic to $\cO_{X^{[b(r-1)-i+j]}}$. 
We prove that $E_{b(r-l),l}$ is isomorphic to 
$\cO_{X^{[b(r-l)+j]}}^l$ by induction. 
Assume that $E_{b(r-l),l} \isom \cO_{X^{[b(r-l)+j]}}^l$. 
Then, for $b(r-l-1) \leq n \leq b(r-l)$, 
consider the extension class $e_{n}$ of the exact sequence 
$$ 0 \lra E_{n,l} \lra E_{n,l+1} \lra \cO_{X^{[n+j]}} \lra 0 $$
in $H^1(X^{[n+j]},E_{n,l}) = H^1(X^{[n+j]},\cO_{X^{[n+j]}})^l$. 
The family of classes $\{e_n\}_n$ defines an element of 
the inverse limit of the diagram 
$$ H^1(X^{[b(r-l)+j]},\cO_{X^{[b(r-l)+j]}})^l \os{F^*}{\lra} 
\cdots 
\os{F^*}{\lra} H^1(X^{[b(r-l-1)+j]},\cO_{X^{[b(r-l-1)+j]}})^l $$
of length $b$ whose last component is $e_{b(r-l-1)}$. 
By twisting by the absolute Frobenius on $k$ 
(which is bijective), we obtain an element of 
the inverse limit of the diagram 
$$ H^1(X^{[b(r-l-1)+j]},\cO_{X^{[b(r-l-1)+j]}})^l \os{F_{\rm abs}^*}{\lra} 
\cdots 
\os{F_{\rm abs}^*}{\lra} H^1(X^{[b(r-l-1)+j]},\cO_{X^{[b(r-l-1)+j]}})^l $$
of length $b$ whose last component is $e_{b(r-l-1)}$. 
Then, by definition of $b$, $e_{b(r-l-1)} = 0$. 
So $E_{b(r-l-1),l+1}$ is isomorphic to $\cO_X^{l+1}$ and we are done. 
\end{proof}

So far, we treated the triviality of sheaves $\cF$ with $\varpi \cF = 0$. 
To lift this triviality to the triviality of certain sheaves $\cF$ 
with $\varpi^N \cF = 0 \, (N \geq 2)$, we need to review the 
definition of the level raising Frobenius pullback functor 
\begin{equation}\label{eq:f}
{\F^s}^*: 
\left( 
\begin{aligned}
& \text{$\cO_{\cX^{[s]}}$-coherent} \\ 
& \text{quasi-nilpotent $\wh{\cD}_{\cX^{[s]}}^{(m)}$-modules} 
\end{aligned} 
\right) 
\os{\cong}{\lra} 
\left( 
\begin{aligned}
& \text{$\cO_{\cX}$-coherent} \\ 
& \text{quasi-nilpotent $\wh{\cD}_{\cX}^{(m+s)}$-modules} 
\end{aligned} 
\right) 
\end{equation}
for $m \geq \fe$. 

\begin{rem}
Precisely speaking, we need 
the level raising Frobenius pullback functor of the form 
\begin{equation}\label{eq:ffff}
{\F^s}^*: 
\left( 
\begin{aligned}
& \text{$\cO_{\cX^{[s+t]}}$-coherent} \\ 
& \text{quasi-nilpotent $\wh{\cD}_{\cX^{[s+t]}}^{(m)}$-modules} 
\end{aligned} 
\right) 
\os{\cong}{\lra} 
\left( 
\begin{aligned}
& \text{$\cO_{\cX^{[t]}}$-coherent} \\ 
& \text{quasi-nilpotent $\wh{\cD}_{\cX^{[t]}}^{(m+s)}$-modules} 
\end{aligned} 
\right) 
\end{equation}
for $m \geq \fe$. However, to lighten the notation, 
we describe the definition in detail only 
in the case of the functor \eqref{eq:f} (the case $t=0$). 
\end{rem}

First we give the definition in local situation. So 
let us forget the projectivity of $X, \cX$ for a while. 
Assume that there exists a local coordinate 
$t_1, ..., t_d$ of $\cX$ over $\Spf V$. 
(Then it induces a local coordinate of 
$\cX^{[s]}$ over $\Spf V$, which we also denote by 
$t_1, ..., t_d$.) Also, take a lift 
$\wt{F}^s: \cX \lra \cX^{[s]}$ of 
$s$-times iteration of the relative Frobenius morphism 
compatible with the morphism $\Spf V \lra \Spf V$ 
induced by $\sigma_V^s$ which satisfies the equalities 
$(\wt{F}^s)^*(t_i) = t_i^{p^s}$ for $1 \leq i \leq d$. 
For $m,n \in \N$ with $m \geq \fe$, 
let $\cD(1)_{(m)}, \cD(1)_{(m)}^n$ 
be as in the previous section. 
Also, let $\cD(1)_{(m)}^{[s]}, \cD(1)_{(m)}^{[s],n}$
be the pull-back of 
$\cD(1)_{(m)}, \cD(1)_{(m)}^n$ by $\pi^n$, respectively. 
Then the homomorphism 
\begin{equation}\label{eq:1}
(\wt{F}^s \times \wt{F}^s)^*: 
\cO_{\cX^{[s]} \times \cX^{[s]}} \lra \cO_{\cX \times \cX} 
\end{equation}
naturally induces the homomorphisms 
\begin{equation}\label{eq:2}
\cO_{\cD(1)_{(m)}^{[s],n}} \lra 
\cO_{\cD(1)_{(m+s)}^n} \qquad (n \in \N). 
\end{equation}
(Here we use the assumption $m \geq \fe$. See 
\cite[2.2.2]{berthelotII}.) 
If we take modulo $p^N$, take the dual, 
take the union with respect to $n$ and take the inverse limit 
with respect to $N \in \N$, 
we obtain the homomorphism 
\begin{equation}\label{eq:3}
\wh{\cD}^{(m+s)}_{\cX} \lra (\wt{F}^s)^*\wh{\cD}^{(m)}_{\cX^{[s]}}. 
\end{equation}
Then, for a $\cO_{\cX^{[s]}}$-coherent 
quasi-nilpotent $\wh{\cD}_{\cX^{[s]}}^{(m)}$-module $\cE$, 
$(\wt{F}^s)^*\cE$ admits an action of 
$\wh{\cD}^{(m+s)}_{\cX}$ via the map \eqref{eq:3} and it 
becomes an $\cO_{\cX}$-coherent 
quasi-nilpotent $\wh{\cD}_{\cX}^{(m+s)}$-module. 
This is the definition of the level raising 
Frobenius pullback functor \eqref{eq:f} in local situation. 
If we put $\tau_i := 1 \otimes t_i - 
t_i \otimes 1$, the homomorphism 
\eqref{eq:1} is written as 
\begin{align}
\tau_i & \mapsto 
1 \otimes t_i^{p^s} - t^{p^s} \otimes 1 \label{eq:4} \\ 
& = (\tau_i + t_i \otimes 1)^{p^s} - t_i^{p^s} \otimes 1 
= \tau_i^{p^s} + \sum_{j=1}^{p^s-1} 
\begin{pmatrix} p^s \\ j \end{pmatrix} 
t_i^{p^s-j} \tau_i^j. \nonumber 
\end{align}

We calculate a part of $\wh{\cD}_{\cX}^{(m+s)}$-action on 
$(\wt{F}^s)^*\cE$ when $m=\fe$ and $\cE$ is torsion. 

\begin{lem}\label{lem:fd1}
Let $N \geq 1$. 
Let the notations be as above with $m=\fe$ 
$($so $\cE$ is an $\cO_{\cX}$-coherent 
quasi-nilpotent $\wh{\cD}_{\cX^{[s]}}^{(\fe)}$-module$)$ and assume that 
$\varpi^N\cE = 0$. Also, let $\{\pa_{\langle l \rangle}\}_{l \in \N^d}$ be 
the family of elements in $\wh{\cD}^{(\fe+s)}_{\cX}$ such that 
the image of 
$\{\pa_{\langle l \rangle}\}_{l \in \N^d, |l| \leq n}$  
in $\cD^{(\fe+s)}_{\cX/p^M}$ is contained in $\cD^{(\fe+s)}_{\cX/p^M,n}$ 
and that it is the dual basis of 
$\{\tau^{\{l\}}\}_{l \in \N^d, |l| \leq n} \subseteq 
\cO_{\cD(1)_{(\fe+s)}^n}/p^M\cO_{\cD(1)_{(\fe+s)}^n}$ for all $M \in \N$. 
$($Here $\cD^{(\fe+s)}_{\cX/p^M}$, $\cD^{(\fe+s)}_{\cX/p^M,n}$ are 
as in the previous section.$)$ Then, for any $j \in \N^d$ with 
$0 < |j| < p^{s-\lfloor \frac{N}{e} \rfloor}$, $\pa_{\langle j \rangle}((\wt{F}^s)^*(x)) 
= 0$ for any $x \in \cE$. 
\end{lem}

\begin{proof}
It suffices to prove that, for any $l \in \N^d$, 
the coefficient of $\tau^j$ in 
the image of $\tau^{\{l\}}$ 
(where ${}^{\{l\}}$ is considered with respect to 
$\fe$-PD structure) by the map \eqref{eq:2} 
is zero modulo $\varpi^N$. 
It suffices to prove it for 
$l$ of the form $(0,...,l_i,...0) \, (l_i > 0)$ for some $i$. 
In this case, the coefficient of $\tau^j$ is zero unless 
$j$ is of the form $(0,...,j_i,...0)$. So 
we can put $l := l_i$, $j := j_i$ 
and regard them as elements in $\N$. 

We estimate the additive $\varpi$-adic valuation $v$ of 
the binomial coefficient 
$\begin{pmatrix} p^s \\ j \end{pmatrix}$, 
assuming $0 < j < p^{s-\lfloor \frac{N}{e} \rfloor}$. 
If we denote by $\alpha(n) \, (n \in \N)$ the sum of 
digits of the $p$-adic expansion of $n$, $v$ is given by 
$v = \frac{e(\alpha(j)+\alpha(p^s-j)-1)}{p-1}$. 
Under the condition $0 < j < p^{s-\lfloor \frac{N}{e} \rfloor}$, 
we have 
$$ p^s > p^s - j > 
 p^{s-\lfloor \frac{N}{e} \rfloor}(p^{\lfloor \frac{N}{e} \rfloor - 1} + 
\cdots + 1)(p-1) $$
and so $\alpha(p^s-j) = (p-1)\lfloor \frac{N}{e} \rfloor + 
\alpha(p^{s-\lfloor \frac{N}{e} \rfloor} - j)$. 
So 
$$ v = e\left( \left\lfloor \frac{N}{e} \right\rfloor + \frac{\alpha(j)+\alpha(p^{s-\lfloor \frac{N}{e} \rfloor} - j)-1}{p-1} 
\right). $$ 
On the right hand side, the second term inside the bracket is equal to or greater than $1$ because it is the $p$-adic additive valuation of 
the binomial coefficient $\begin{pmatrix} p^{s-\lfloor \frac{N}{e} \rfloor} \\ j \end{pmatrix}$. So we have the estimate 
$v \geq e(\left\lfloor \frac{N}{e} \right\rfloor +1) \geq N.$ 
From this and the calculation \eqref{eq:4}, we see that 
the coefficient of $\tau_i^j$ in 
the image of $\tau_i$ by the map \eqref{eq:2} 
is zero modulo $\varpi^N$, which we denote by $\varpi^Nc_j$. 
Then the coefficient of $\tau_i^j$ in 
the image of $\tau_i^{\{l\}}$ by the map \eqref{eq:2} 
is equal to 
$$ 
\frac{\varpi^{Nl}}{q_l!} 
\sum_{\substack{j_1 + \cdots + j_l = j \\ j_i > 0}} c_{j_1} \cdots c_{j_l}, 
$$
where $q_l := \lfloor \frac{l}{p^{\fe}} \rfloor$. 
We should prove that it is zero modulo $\varpi^N$. So it suffices to 
prove that the additive $\varpi$-adic valuation $w$ of 
$\frac{\varpi^{Nl}}{q_l!}$ is equal to or greater than $N$, 
which follows from the calculation 
\begin{align*}
w & = 
Nl - \frac{e(q_l - \alpha(q_l))}{p-1} 
> Nl - \frac{el}{p^{\fe}(p-1)}
\geq Nl - l \geq N-1. 
\end{align*}
So we are done. 
\end{proof}

Next we explain the definition of the level raising 
Frobenius pullback functor \eqref{eq:f} in global situation. 
So let $X, \cX$ be projective again. 
Let us take an open covering $\cX = \bigcup_{\alpha} \cX_{\alpha}$ 
of $\cX$ such that each $\cX_{\alpha}$ admits 
a local coordinate and a lift $\wt{F}_{\alpha}^{s}$ 
of $s$-times iteration of relative Frobenius 
as in the local situation. Then 
the definition in local situation says that, 
for an $\cO_{\cX^{[s]}}$-coherent 
quasi-nilpotent $\wh{\cD}_{\cX^{[s]}}^{(m)}$-module $\cE$, 
$(\wt{F}_{\alpha}^s)^*\cE$ has a structure of 
$\cO_{\cX_{\alpha}}$-coherent 
quasi-nilpotent $\wh{\cD}_{\cX_{\alpha}}^{(m+s)}$-module. 
So it suffices to glue this local definition. 
Let us put $\cX_{\alpha\beta} := 
\cX_{\alpha} \cap \cX_{\beta}$. Let 
$\cD(1)_{(m),\alpha\beta}$ be the open formal subscheme of 
$\cD(1)_{(m)}$ which is homeomorphic to $\cX_{\alpha\beta}$ and 
let $p_i: \cD(1)_{(m), \alpha\beta} \lra \cX_{\alpha\beta} \, 
(i = 1,2)$ be projections. 
Also, let ${}^{[s]}$ denote the pullback of formal schemes by 
$\pi^s$. Then it is known that the morphism 
$$\wt{F}_{\beta}^s|_{\cX_{\alpha\beta}} \times 
\wt{F}_{\alpha}^s|_{\cX_{\alpha\beta}}: \cX_{\alpha\beta} \lra 
\cX^{[s]}_{\alpha\beta} \times_V \cX^{[s]}_{\alpha\beta}$$
factors through $\cD(1)_{(m),\alpha\beta}^{[s]}$. 
(Here we use the assumption $m \geq \fe$.) 
The structure of $\wh{\cD}_{\cX^{[s]}}^{(m)}$-module on $\cE$ induces that of 
an HPD-stratification, which induces an isomorphism 
$\epsilon: p_2^*\cE \lra p_1^*\cE$ on $\cD(1)_{(m),\alpha\beta}^{[s]}$. 
By pulling it back to $\cX_{\alpha\beta}$ via 
the morphism induced by 
$\wt{F}_{\beta}^s|_{\cX_{\alpha\beta}} \times 
\wt{F}_{\alpha}^s|_{\cX_{\alpha\beta}}$, 
we obtain the isomorphism 
$\wt{F}_{\alpha}^*\cE \lra \wt{F}_{\beta}^*\cE$, and we can check 
that this gives the glueing data. So we obtain 
the level raising Frobenius pullback ${\F^s}^*\cE$. 
(In fact, it is known that the isomorphism above is an isomorphism 
of $\wh{\cD}^{(m+s)}_{\cX_{\alpha\beta}}$-modules and so 
${\F^s}^*\cE$ has a structure of $\wh{\cD}^{(m+s)}_{\cX}$-module.) 

Let $\cE$ be a $V/\varpi^{N+1} V$-flat, 
coherent $\cO_{\cX^{[s]}}/\varpi^{N+1}\cO_{\cX^{[s]}}$-module 
such that $\cE/\varpi^N\cE = 
(\cO_{\cX^{[s]}}/\varpi^N\cO_{\cX^{[s]}})^r$. Then $\cE$ is a 
locally free $\cO_{\cX^{[s]}}/\varpi^{N+1}\cO_{\cX^{[s]}}$-module of rank $r$. 
Take a sufficiently fine open covering $\cX^{[s]} = 
\bigcup_{\alpha} \cX^{[s]}_{\alpha}$ 
of $\cX$ and fix an isomorphism 
$\cE|_{\cX_{\alpha}^{[s]}} \cong (\cO_{\cX_{\alpha}^{[s]}}/
\varpi^{N+1}\cO_{\cX^{[s]}_{\alpha}})^r$ 
which lifts the fixed equality $\cE/\varpi^N\cE = 
(\cO_{\cX^{[s]}}/\varpi^N\cO_{\cX^{[s]}})^r$. 
Then, on each $\cX_{\alpha\beta}^{[s]}$, we have an isomorphism 
$$ 
(\cO_{\cX_{\alpha\beta}^{[s]}}/\varpi^{N+1}\cO_{\cX^{[s]}_{\alpha\beta}})^r 
\cong 
(\cE|_{\cX_{\alpha}^{[s]}})|_{\cX_{\alpha\beta}^{[s]}} 
= (\cE|_{\cX_{\beta}^{[s]}})|_{\cX_{\alpha\beta}^{[s]}} 
 \cong (\cO_{\cX_{\alpha\beta}^{[s]}}/
\varpi^{N+1}\cO_{\cX_{\alpha\beta}^{[s]}})^r $$
which lifts the identity on 
$(\cO_{\cX^{[s]}_{\alpha\beta}}/\varpi^{N}\cO_{\cX_{\alpha\beta}^{[s]}})^r$. 
So this map has the form 
$1 + \varpi^N\Lambda_{\alpha\beta}$, where 
$\Lambda_{\alpha\beta} \in \Gamma(X_{\alpha\beta}^{[s]}, M_n(\cO_{X^{[s]}}))$. 
(Here $X^{[s]}_{\alpha\beta} := \cX^{[s]}_{\alpha\beta} \otimes_V k$.) 
Then $e(\cE) := \{\Lambda_{\alpha\beta}\}$ defines an element of 
$H^1(X^{[s]}, M_n(\cO_{\cX^{[s]}})) = H^1(X^{[s]},\cO_{X^{[s]}})^{r^2}$, and 
it is trivial if and only if $\cE$ is isomorphic to 
 $(\cO_{\cX^{[s]}}/\varpi^{N+1}\cO_{\cX^{[s]}})^r$. 
The next proposition calculates $e(\F^*\cE)$ when 
$\cE$ has a structure of quasi-nilpotent $\wh{\cD}_{\cX^{[s]}}^{(m)}$-module 
for sufficiently large $m$: 

\begin{prop}\label{prop:c}
Let $N \geq 1$. 
Let $\cE$ a $V/\varpi^{N+1}V$-flat, coherent $\cO_{\cX^{[s]}}/\varpi^{N+1}
\cO_{\cX^{[s]}}$-module 
endowed with 
a structure of quasi-nilpotent $\wh{\cD}_{\cX^{[s]}}^{(m)}$-module 
such that $\cE/\varpi^N\cE = (\cO_{\cX^{[s]}}/\varpi^N \allowbreak \cO_{\cX^{[s]}})^r$ as 
$\cO_{\cX^{[s]}}$-modules, and 
let $\F^*\cE$ be its image by the level raising Frobenius pullback. 
Then, if $m \geq 2N+ \fe + 1$, $e(\F^*\cE) = F^*e(\cE)$, where 
$F^*$ on the right hand side is the map 
$H^1(X^{[s]},\cO_{X^{[s]}})^{r^2} \lra H^1(X^{[s-1]},\cO_{X^{[s-1]}})^{r^2}$
 induced by relative Frobenius. 
\end{prop}

\begin{proof}
Let us take 
a sufficiently fine open covering $\cX^{[s-1]} = 
\bigcup_{\alpha} \cX_{\alpha}^{[s-1]}$ 
of $\cX^{[s-1]}$ as above such that $\cX_{\alpha}^{[s-1]}$ admits 
a local coordinate and a lift of relative Frobenius 
$\wt{F}_{\alpha}: \cX_{\alpha}^{[s-1]} \lra \cX_{\alpha}^{[s]}$ 
as in the local definition of level raising Frobenius pullback. 
Then we have the isomorphism 
$\cE|_{\cX_{\alpha}^{[s]}} \cong (\cO_{\cX^{[s]}_{\alpha}}/\varpi^{N+1}
\cO_{\cX_{\alpha}^{[s]}})^r$ 
as above, and $\F^*\cE$ is defined by glueing 
$(\cO_{\cX_{\alpha}^{[s-1]}}/\varpi^{N+1}\cO_{\cX_{\alpha}^{[s-1]}})^r \cong 
\wt{F}_{\alpha}^*\cE$ via the composite 
$$
(\cO_{\cX_{\alpha\beta}^{[s-1]}}/\varpi^{N+1}\cO_{\cX^{[s-1]}_{\alpha\beta}})^r \cong 
(\wt{F}_{\alpha}^*\cE)|_{\cX_{\alpha\beta}^{[s-1]}}
\cong 
(\wt{F}_{\beta}^*\cE)|_{\cX_{\alpha\beta}^{[s-1]}}
\cong 
(\cO_{\cX_{\alpha\beta}^{[s-1]}}/
\varpi^{N+1}\cO_{\cX^{[s-1]}_{\alpha\beta}})^r, 
$$ 
where the isomorphism in the middle comes from the structure of 
HPD-stratification on $\cE$. Hence, an element 
$\wt{F}_{\alpha}^*(\e) \in 
(\cO_{\cX_{\alpha\beta}^{[s-1]}}/
\varpi^{N+1}\cO_{\cX^{[s-1]}_{\alpha\beta}})^r$ 
is sent by this isomorphism to 
\begin{equation}\label{eq:100}
(1 + \varpi^NF^*(\Lambda_{\alpha\beta})) \left(\sum_{l \in \N^d} (\wt{F}_{\beta}^*(t)-\wt{F}_{\alpha}^*(t))^{\{l\}} 
\wt{F}_{\alpha}^* (\pa_{\langle l \rangle}(\e)) \right),  
\end{equation}
where the local coordinate $t$ and the differential operators 
$\pa_{\langle l \rangle}$ are the ones we took on $\cX_{\alpha}^{[s]}$ and 
in $\cD^{(m)}_{\cX_{\alpha}^{[s]}}$, respectively.  

Since $\cE$ is a $\varpi^{N+1}$-torsion $\cO_{\cX^{[s]}}$-coherent
 quasi-nilpotent $\wh{\cD}_{\cX^{[s]}}^{(m)}$-module, 
it is written as the image of a 
$\varpi^{N+1}$-torsion $\cO_{\cX^{[s+m-\fe]}}$-coherent
quasi-nilpotent $\wh{\cD}_{\cX^{[s+m-\fe]}}^{(\fe)}$-module by 
${\F^{m-\fe}}^*$. Hence, by Lemma \ref{lem:fd1}, 
$\pa_{\langle l \rangle}$ acts on some basis $B$ of 
$\cE |_{\cX_{\alpha\beta}^{[s]}}$ 
by zero when $0 < |l| < p^{m-\fe-\lfloor \frac{N+1}{e} \rfloor}$. 
So, when $\e$ runs through the above basis $B$ of $\cE 
 |_{\cX_{\alpha\beta}^{[s]}}$, 
the terms which may survive in the big bracket in 
\eqref{eq:100}
are the constant term $\wt{F}_{\alpha}^*(\e)$ and 
the terms with 
$(\wt{F}_{\alpha}^*(t)-\wt{F}_{\beta}^*(t))^{\{l\}}$, $|l| \geq 
p^{m-\fe-\lfloor \frac{N+1}{e} \rfloor}$. 
Because $\wt{F}_{\alpha}^*(t)-\wt{F}_{\beta}^*(t)$ is divisible by $\varpi$, 
$(\wt{F}_{\alpha}^*(t)-\wt{F}_{\beta}^*(t))^{\{l\}}$ is contained 
in $p^c\cE$, where $c$ is the additive $\varpi$-adic valuation 
of $\varpi^{|l|}/q_{|l|}!$, where 
$q_{|l|} = \lfloor |l|/p^m \rfloor$. 
Then if we denote by $\alpha(x) \, (x \in \N)$ the sum of 
digits of the $p$-adic expansion of $x$, we have the 
following estimate for $c$: 
\begin{align*}
c & = |l| - \frac{e(q_{|l|}-\alpha(q_{|l|}))}{p-1} 
> |l| - \frac{e|l|}{p^m(p-1)} 
\geq p^{m-\fe-\lfloor \frac{N+1}{e} \rfloor}
\left(1  - \frac{e}{p^m(p-1)} \right). 
\end{align*}
Then, when $m \geq 2N+ \fe + 1$, 
$m-\fe-\lfloor \frac{N+1}{e} \rfloor \geq N$ and 
$m \geq \fe + 3$. So 
$$
 p^{m-\fe-\lfloor \frac{N+1}{e} \rfloor}
\left(1  - \frac{e}{p^m(p-1)} \right) 
\geq 2^{N} \left( 1 - \frac{1}{8} \right) \geq N. $$ 
So $c \geq N+1$ and 
\eqref{eq:100} is equal to $(1 + \varpi^NF^*(\Lambda_{\alpha\beta}))(\e)$ 
when $\e$ runs through the basis $B$ of $\cE |_{\cX_{\alpha\beta}^{[s]}}$. 
Thus we see that 
$e(\F^*\cE) = \{F^*(\Lambda_{\alpha\beta})\} = 
F^*e(\cE)$. 
\end{proof}

\begin{cor}\label{cor:c}
Let $\cE$ be as in the proposition above with $s \geq b$, where 
$b$ is as in Proposition \ref{prop:b}. Then 
${\F^b}^*\cE$ is isomorphic to 
$(\cO_{\cX^{[s-b]}}/\varpi^{N+1}\cO_{\cX^{[s-b]}})^r$ as 
$\cO_{\cX^{[s-b]}}$-modules. 
\end{cor}

\begin{proof}
By Proposition \ref{prop:c}, 
$e({\F^b}^*\cE) = {F^b}^*e(\cE)$ and by definition of 
$b$ given in the proof of Proposition \ref{prop:b}, it is zero. 
So ${F^b}^*\cE$ is isomorphic to 
$(\cO_{\cX^{[s-b]}}/\varpi^{N+1}\cO_{\cX^{[s-b]}})^r$ as 
$\cO_{\cX^{[s-b]}}$-modules.
\end{proof}

Now we give a proof of Theorem \ref{thm:1}. 

\begin{proof}[Proof of Theorem \ref{thm:1}] 
Let $r \in \N$ be the rank of $\cE$ 
and take $a, b \in \N$ such that the conclusion of 
Propositions \ref{prop:a}, \ref{prop:b} are satisfied. 
By making $a, b$ larger, we may assume that 
$a \geq b \geq \fe + 3$. 

First, note that, for any $m' \leq m$ and for any $i$, 
a quasi-nilpotent 
$\wh{\cD}_{\cX^{[i]}}^{(m)}$-module can be regardrd also 
as a quasi-nilpotent 
$\wh{\cD}_{\cX^{[i]}}^{(m')}$-module. Using this, 
we can replace the infinite set of $m \in \N$ for which 
$\cG^{[i](m)}$ is defined by any infinite subset of $\N$. 
So us put 
$I := \{m \,|\, b | (m-a-\fe), m \geq a + \fe\}$ and 
assume in the following that $\cG^{[i](m)}$ is defined for any 
$m \in I$. Now take 
any $m = b(m'-1)+a+\fe \in I$. 
For $0 \leq j \leq m-\fe$, 
let $\cG^{[i](m)[j]}$ be the $j$-th Frobenius antecedent of 
$\cG^{[i](m)}$ and put 
$G^{[i](m)[j]} := \cG^{[i](m)[j]}/\varpi\cG^{[i](m)[j]}$. 
Since $\cG^{[i](m)[j]}$ is $p$-torsion free, 
$\mu(G^{[i](m)[j]}) = \mu(\bE^{[i+j]}) = 0$, 
$p_{G^{[i](m)[j]}} = p_{\bE^{[i+j]}} = p_{\cO_X}$. 
Also, $G^{[i](m)[j]}$ is $\mu$-stable (hence stable): 
Indeed, for any coherent subsheaf $0 \not= H \subsetneq G^{[i](m)[j]}$, 
$p^{j} \mu(H) = \mu({F^j}^*H) < \mu(G^{[i](m)}) = 0$. 
(Here the inequality follows from the 
$\mu$-stability of $G^{[i](m)}$.) 
Hence any subsequence of length $a$ of 
the sequence $\{G^{[i](m)[j]}\}_{j=0}^{m-\fe}$ satisfies 
the assumption of Proposition \ref{prop:a}. Hence 
the sequence $\{G^{[i](m)[j]}\}^{b(m'-1)}_{j=0}$ 
is the constant sequence $\{\cO_{X^{[i+j]}}\}^{j=b(m'-1)}_{j=0}$. 
(So $r$ should be $1$.) 

In particular, $G^{[i](m)[b(m'-1)]} = 
\cG^{[i](m)[b(m'-1)]}/\varpi\cG^{[i](m)[b(m'-1)]}$ 
is isomorphic to $\cO_X = \cO_{\cX}/\varpi\cO_{\cX}$ as 
$\cO_{\cX}$-modules. Because it has a structure of 
$\wh{\cD}^{(a)}_{\cX}$-module, it has a structure of 
$\wh{\cD}^{(b)}_{\cX}$-module. 
We prove by induction that 
$\cG^{[i](m)[b(m'-l)]}/\varpi^l\cG^{[i](m)[b(m'-l)]}$
is isomorphic to $\cO_{\cX^{[i+b(m'-1)]}}/\varpi^l\cO_{\cX^{[i+b(m'-1)]}}$ 
as $\cO_{\cX^{[i+b(m'-1)]}}$-modules and 
it has a structure of $\wh{\cD}^{(lb)}_{\cX^{[i+b(m'-1)]}}$-module. 
Indeed, assume that this is true for $l$. 
Then, since $lb \geq l(\fe+3) \geq 2l + \fe + 1$, 
$${\F^b}^*(\cG^{[i](m)[b(m'-l)]}/\varpi^{l+1}\cG^{[i](m)[b(m'-l)]}) 
= \cG^{[i](m)[b(m'-l-1)]}/\varpi^{l+1}\cG^{[i](m)[b(m'-l-1)]}$$ 
is isomorphic to 
$\cO_{\cX^{[i+b(m'-l-1)]}}/\varpi^{l+1}\cO_{\cX^{[i+b(m'-l-1)]}}$ 
as $\cO_{\cX^{[i+b(m'-l-1)]}}$-modules by Corollary \ref{cor:c}, 
and it has a structure of $\wh{\cD}^{((l+1)b)}_{\cX^{[i+b(m'-l-1)]}}$-modules. 
So we see finally that 
$\cG^{[i](m)}/\varpi^{m'}\cG^{[i](m)}$ is 
 isomorphic to $\cO_{\cX^{[i]}}/\varpi^{m'}\cO_{\cX^{[i]}}$ as 
$\cO_{\cX^{[i]}}$-modules. 

Now let us move $m$: Put $i_0 := \min_{m \in I} (i(m))$ and 
take $m_0 \in I$ with $i_0 = i(m_0)$. First we check that 
$\bE^{[i_0]}$ is $\mu$-stable. 
Let $\fG$ be the coherent $\cO_{\fX}$-module corresponding to 
$\cG^{[i_0](m_0)}$ via GFGA. Then, 
for any coherent subsheaf $0 \not= \bE' \subsetneq \bE^{[i_0]}$, 
$\fG':= \fG \cap (\alpha_* \bE')$ (where 
$\alpha: \bX \lra \fX$ is the canonical open immersion) 
is a $p$-torsion free coherent subsheaf of $\fG$ such that $\fG/\fG'$
 is again $p$-torsion free. Hence $\fG'/\varpi\fG'$ is a coherent 
subsheaf of $\fG/\varpi\fG = G^{[i_0](m_0)}$ and so 
$\mu(\bE') = \mu(\fG'/\varpi\fG') < \mu(G^{[i_0](m_0)}) 
= \mu(\bE^{[i_0]})$ because $G^{[i_0](m_0)}$ is $\mu$-stable. 

Next, for $m = b(m'-1) + a + \fe \in I$, let $\cH^{\langle m \rangle}$ 
be the level raising Frobenius pullback 
of $\cG^{[i](m)}$ by $(i-i_0)$-times iteration of 
relative Frobenius. Then 
$\cH^{\langle m \rangle}/\varpi^{m'}\cH^{\langle m \rangle}$ is isomorphic 
to $\cO_{\cX^{[i_0]}}/\varpi^{m'}\cO_{\cX^{[i_0]}}$ 
and so 
$\cH^{\langle m \rangle}/\varpi\cH^{\langle m \rangle}$ is also $\mu$-stable for any 
$m \in I$. 

Hence, by theorem of Langton \cite{langton}, 
$\cH^{\langle m \rangle}$'s have the form 
$\varpi^{c_m}\cH^{\langle m_0 \rangle}$ for some $c_m \in \Z$ depending on $m$. 
This implies that 
$\cH^{\langle m_0 \rangle}/\varpi^{m'}\cH^{\langle m_0 \rangle}$ is isomorphic to 
$\cH^{\langle m \rangle}/\varpi^{m'}\cH^{\langle m \rangle}$, thus to 
$\cO_{\cX^{[i_0]}}/\varpi^{m'}\cO_{\cX^{[i_0]}}$ for all $m \in I$. Therefore, 
$\cH^{\langle m_0 \rangle}$ is isomorphic to 
$\cO_{\cX^{[i_0]}}$. 
So $\cE^{[i_0]} = \Q \otimes_{\Z} \cH^{\langle m_0 \rangle}$ is 
isomorphic to $\Q \otimes_{\Z} \cO_{\cX^{[i_0]}}$. 
To show the triviality of the action of $\cD^{\dagger}_{\cX^{[i_0]},\Q}$ on 
$\cE^{[i_0]}$, it suffices to see that 
$H^0(\cX^{[i_0]}, \Q \otimes \Omega^1_{\cX^{[i_0]}}) 
= H^0(\bX^{[i_0]},\Omega^1_{\bX^{[i_0]}}) = 0$. 
This follows from the (in)equalities
$$ 
\dim H^0(\bX^{[i_0]},\Omega^1_{\bX^{[i_0]}}) \leq 
\dim H^1_{\dR}(\bX^{[i_0]}) = \dim H^1_{\et}(\ol{\bX}^{[i_0]}, \Q_l) = 
\dim H^1_{\et}(X^{[i_0]},\Q_l) = 0. 
$$
Hence $\cE^{[i_0]}$ is a trivial convergent isocrystal and so 
is $\cE$. 
\end{proof}

Next we give a proof of Theorem \ref{thm:2}. 

\begin{proof}[Proof of Theorem \ref{thm:2}] 
Let $r, a, b$ as in the proof of \ref{thm:1}. 

First, by the same reason as before, 
we can replace the infinite set of $m \in \N$ for which 
$\cG^{[i](m)}$ is defined by any infinite subset of $\N$. 
Also, we can replace $l(m) (\leq m - \fe)$ for each $m$ 
by any smaller value, keeping the property $l(m) \to \infty 
\, (m \to \infty)$. 
Using this, we see that we may take an infinite subset $I'$ of $\N$, the set 
$I := \{c \,|\,\, br | (c-ar), c \geq br(r-1)+ar \}$ and 
a bijection $l: I' \lra I$ so that 
$\cG^{[i](m)}$ is defined for any 
$m \in I'$ and that $l = l(m)$ in the statement of 
Theorem \ref{thm:2} is given by the image of $m$ by 
the map $l$ above. 

Now take any 
$m \in I'$ so that $l := l(m) =  br(m'+r-2)+ar$. 
For $0 \leq j \leq l$, 
define the Jordan-H\"older filtration \cite[Definition 1.5.1]{hl} 
$\{U^{[i](m)[j]}_q\}_{q=0}^{q_j}$ of $G^{[i](m)[j]}$ 
in the following way, by descending induction: 
First, when $j=l$, take 
any Jordan-H\"older filtration 
$\{U^{[i](m)[l])}_q\}_{q=0}^{q_l}$ of $G^{[i](m)[l]}$. 
When we defined 
$\{U^{[i](m)[j+1]}_q\}_{q=0}^{q_{j+1}}$, 
the pull-back $\{F^*U^{[i](m)[j+1]}_q\}_{q=0}^{q_{j+1}}$ of it 
by relative Frobenius defines a filtration of 
$F^*G^{[i](m)[j+1]} = G^{[i](m)[j]}$ whose graded pieces are 
semistable. Then we define $\{U^{[i](m)[j]}_q\}_{q=0}^{q_j}$ 
by taking any Jordan-H\"older filtration which refines 
$\{F^*U^{[i](m)[j+1]}_q\}_{q=0}^{q_{j+1}}$. 
By definition, we have 
$$ r \geq q_0 \geq q_1 \geq \cdots \geq q_l \geq 1, $$
where $r$ is the rank of $\cE$. 
So, there exists some 
$j_0$ such that 
$q_{j_0} = \cdots = q_{j_0+b(m'+r-2)+a} \allowbreak (=:Q)$. 

Put $V^{[i](m)[j]}_q := U^{[i](m)[j]}_q/U^{[i](m)[j]}_{q-1}$. Then, 
for each $1 \leq q \leq Q$, 
any subsequence of length $a$ of 
the sequence $\{V^{[i](m)[j]}_q\}^{j=j_0+b(m'+r-2)+a}_{j=j_0}$ 
satisfies 
the assumption of Proposition \ref{prop:a}. Hence 
$\{V^{[i](m)[j]}_q\}^{j=j_0+b(m'+r-1)}_{j=j_0}$ 
is isomorphic to 
the constant sequence $\{\cO_{X^{[i+j]}}\}^{j_0+b(m'+r-2)}_{j=j_0}$. 
Then, we can apply Proposition \ref{prop:b} to 
any subsequence of length $b(r-1)$ of
the sequence $\{G^{[i](m)[j]}\}^{j=j_0+b(m'+r-2)}_{j_0}$. 
So $\{G^{[i](m)[j]}\}^{j_0+b(m'-1)}_{j=j_0}$ is 
isomorphic to the constant sequence 
$\{\cO_{X^{[i+j]}}^r\}^{j_0+b(m'-1)}_{j=j_0}$. 
Then, by the same argument as the proof of Theorem \ref{thm:1}, 
we see from Proposition \ref{prop:c} that 
$\cG^{[i](m)}/\varpi^{m'}\cG^{[i](m)}$ is isomorphic to 
$(\cO_{\cX^{[i]}}/\varpi^{m'}\cO_{\cX^{[i]}})^r$. 

Now let us move $m$: Put $i_0 := \min_{m \in I'} (i(m))$, and 
for each $m \in I'$ with 
$l(m) = br(m'+r-2)+ar \in I$, let $\cH^{\langle m \rangle}$ 
be the level raising Frobenius pullback 
of $\cG^{[i](m)}$ by $(i-i_0)$-times iteration of 
relative Frobenius. 
Note that $\cH^{\langle m \rangle}/\varpi^{m'}\cH^{\langle m \rangle}$ is 
isomorphic to $(\cO_{\cX^{[i_0]}}/\varpi^{m'}\cO_{\cX^{[i_0]}})^r$ and so 
$\cH^{\langle m \rangle}/\varpi\cH^{\langle m \rangle}$ is semistable. 
On the other hand, 
$\bE^{[i_0]}$ is also semistable. 
(This can be proven in the same way as the proof of Theorem \ref{thm:1}.) 

Let $\ol{M}$ be the moduli 
of semistable sheaves on $\fX^{[i_0]}$ with rank $r$ and reduced 
Hilbert polynomial 
$p_{\cO_X}$, which is constructed by Adrian Langer (\cite{al1}, \cite{al2}). 
It is 
a projective scheme over $\Spec V$. Then, for any $m \in I$, 
$\cH^{\langle m \rangle}$ defines a $V$-valued point $P_{m}$ 
of $\ol{M}$ which induces the $K$-valued point $P_K$ defined by 
by $\bE^{[i_0]}$. Since $\ol{M}$ is separated, the $V$-point 
which extends $P_K$ is unique. Hence $P_{m}$ is independent of 
$m$, which we denote by $P$. 
(This does not imply that $\cH^{\langle m \rangle}$ are 
independent of $m$ because $\ol{M}$ is not a fine moduli.) 
On the other hand, let $P'$ be the $V$-valued point defined by 
$\cO_{\fX^{[i_0]}}^r$. Since 
$\cH^{\langle m \rangle}/\varpi^{m'}\cH^{\langle m \rangle}$ is trivial, 
the $V/\varpi^{m'}V$-valued point induced by $P$ is the same as 
the $V/\varpi^{m'}V$-valued point induced by $P'$. So $P=P'$ and this implies 
that $P_K$ is equal to the $K$-valued point defined by 
$\cO_{\bX^{[i_0]}}^r$. Hence $\bE^{[i_0]}$ is $S$-equivalent 
\cite[Definition 1.5.3]{hl} to 
$\cO_{\bX^{[i_0]}}^r$ when pulled back to the geometric fiber 
$\ol{\bX}^{[i_0]}$, namely, 
$\bE^{[i_0]}|_{\ol{\bX}^{[i_0]}}$ is an iterated extension of 
$\cO_{\ol{\bX}^{[i_0]}}$. Since $H^1(\ol{\bX}^{[i_0]},\cO_{\ol{\bX}^{[i_0]}}) 
\subseteq H^1_{\dR}(\ol{\bX}^{[i_0]}) = 0$, 
$\bE^{[i_0]}|_{\ol{\bX}^{[i_0]}}$ is isomorphic to 
$\cO_{\ol{\bX}^{[i_0]}}^r$, and this imples that 
$\bE^{[i_0]}$ is isomorphic to $\cO_{\bX^{[i_0]}}^r$. 
Then, by the same method as the proof of 
Theorem \ref{thm:1}, we see that 
$\cE^{[i_0]}$ is a trivial convergent isocrystal and so 
is $\cE$. 
\end{proof}

\end{document}